\documentclass[12pt,a4paper]{amsart}
\usepackage{mathrsfs}

\newtheorem{theo+}              {Theorem}           [section]
\newtheorem{prop+}  [theo+]     {Proposition}
\newtheorem{coro+}  [theo+]     {Corollary}
\newtheorem{lemm+}  [theo+]     {Lemma}
\newtheorem{exam+}  [theo+]     {Example}
\newtheorem{rema+}  [theo+]     {Remark}
\newtheorem{defi+}  [theo+]     {Definition}
\newtheorem{clai+}  [theo+]     {Claim}
\newenvironment{theorem}{\begin{theo+}}{\end{theo+}}
\newenvironment{proposition}{\begin{prop+}}{\end{prop+}}
\newenvironment{corollary}{\begin{coro+}}{\end{coro+}}

\usepackage{amsthm}
\theoremstyle{plain} \theoremstyle{remark}
\newtheorem{remark}{Remark}
\newtheorem{example}{Example}

\def \r{\mbox{${\mathbb R}$}}
\def\E{/\kern-1.0em \equiv }

\author{Ze-Ping Wang$^{*}$, Ye-Lin Ou$^{**}$ and Qi-Long Liu $^{*}$}

\address{*Department of Mathematics,\newline\indent Guizhou
Normal University,\newline\indent Guiyang 550025,\newline\indent
People's Republic of China
\newline\indent E-mail:zpwzpw2012@126.com \;(Wang)\\\newline\indent liuqilong2008.ok@163.com (Liu)\\\newline\indent  \\\newline\indent Department of Mathematics,\newline\indent Texas A $\&$ M University-Commerce,
\newline\indent Commerce TX 75429,\newline\indent USA.\newline\indent
E-mail:yelin$\_$ou@tamu-commerce.edu \;(Ou)}
\thanks{*Supported by the Natural Science Foundation of China (No. 11861022). \\
\indent** Supported by a grant from the Simons Foundation ( 427231,
Ye-Lin Ou).}
\date{10/06/2022}
\begin{document}

\title[Harmonic and biharmonic Riemannain submersions ] {Harmonic and biharmonic Riemannain submersions from Sol space }

\subjclass{58E20, 53C12} \keywords{Harmonic map, Biharmonic maps, Riemannain submersions,  Sol space.}

\maketitle

\section*{Abstract}
\begin{quote}
{\footnotesize In this paper,  we give a complete classification of harmonic and biharmonic Riemannian submersions $\pi:(\r^3,g_{Sol})\to (N^2,h)$ from Sol space into
a surface by proving that there is neither harmonic nor biharmonic Riemannian submersion  $\pi:(\r^3,g_{Sol})\to (N^2,h)$ from Sol space
no matter what the base space $(N^2,h)$ is.  We also prove that a Riemannian submersion $\pi:(\r^3,g_{Sol})\to (N^2,h)$ from Sol space exists only when the base space is a hyperbolic space form.}
\end{quote}

\section{Introduction and Preliminaries}
All manifolds, maps, tensor fields studied in this paper are
assumed
to be smooth unless there is an otherwise statement.\\

Recall that a {\em harmonic map} $\varphi:(M, g)\to (N,
h)$ between Riemannian manifolds is a critical point of the energy functional
\begin{equation}\nonumber
E\left(\varphi,\Omega \right)= \frac{1}{2} {\int}_{\Omega}
\left|{\rm d}\varphi \right|^{2}{\rm d}x.
\end{equation}
 The Euler-Lagrange equation is given by the vanishing
of the tension filed $\tau(\varphi)={\rm
Trace}_{g}\nabla {\rm d} \varphi$ (see \cite{EL1}). Clearly,  the  map $\varphi$ is harmonic if and only if
$\tau(\varphi)={\rm Trace}_{g}\nabla {\rm d} \varphi=0$ holds identically.\\

The study of biharmonic maps as a special case of $k$-polyharmonic
maps were first  proposed   by J. Eells and L. Lemaire  in \cite{EL1}. A {\em biharmonic map}  $\varphi:(M, g)\to (N,
h)$ between Riemannian manifolds  is  a critical point of the
bienergy
\begin{equation}\nonumber
E^{2}\left(\varphi,\Omega \right)= \frac{1}{2} {\int}_{\Omega}
\left|\tau(\varphi) \right|^{2}{\rm d}x
\end{equation}
for every compact subset $\Omega$ of $M$, where $\tau(\varphi)={\rm
Trace}_{g}\nabla {\rm d} \varphi$ is the tension field of $\varphi$.  Jiang \cite{Ji} first computed the first
variation of the functional to  see that $\varphi$
is  biharmonic  if and only if its bitension field vanishes
identically, i.e.,
\begin{equation}\label{BT1}
\tau^{2}(\varphi):={\rm
Trace}_{g}(\nabla^{\varphi}\nabla^{\varphi}-\nabla^{\varphi}_{\nabla^{M}})\tau(\varphi)
- {\rm Trace}_{g} R^{N}({\rm d}\varphi, \tau(\varphi)){\rm d}\varphi
=0,
\end{equation}
where $R^{N}$ is the curvature operator of $(N, h)$ defined by
$$R^{N}(X,Y)Z=
[\nabla^{N}_{X},\nabla^{N}_{Y}]Z-\nabla^{N}_{[X,Y]}Z.$$

 We call a submanifold that is a biharmonic submanifold if the
isometric immersion that defines the  submanifold is a biharmonic
map. Analogously, a  Riemannian submersion is called a {\bf  biharmonic (respectively, harmonic) Riemannian
submersion} if the Riemannian submersion is a  biharmonic (respectively, harmonic) map. Obviously,  any harmonic map  is always biharmonic whilst biharmonic maps include harmonic maps as special cases. We use proper biharmonic maps
(respectively, submanifolds, Riemannian submersion)  to call those biharmonic maps that are not harmonic maps.\\

For harmonicity of Riemannian submersions,  one of our motivations is that the definition of Riemannian submersions, in a sense, are considered as the dual notion of  isometric immersions (i.e., submanifolds). There are many interesting examples of harmonic isometric immersions of a surface (i.e., minimal surfaces) into 3-manifolds, such as planes or catenoid in $\r^3$ or  harmonic embedding of $S^2$ into $S^3$ \cite{Sm}. On the other hand, there exist many interesting examples and classification results of  harmonic Riemannian submersions  from 3-dimensional Riemannian manifolds into a surface:
 Hopf fibration $\pi: S^3\to S^2(4)$ and the orthogonal projection $\pi: \r^3\to \r^2$ are harmonic Riemannian submersion;  there is no harmonic Riemannian submersion $\pi:H^3\to(N^2,h)$ no matter what $(N^2,h)$ is (see \cite{WO,WO4}); harmonic Riemannian submersions from Thurston's 3-dimensional geometries, 3-dimensional BCV spaces and a Berger sphere $S_{\varepsilon}^3$ have been completely classified and many explicit constructions of harmonic Riemannian submersions  were given (see \cite{WO4} for details).\\

   Since biharmonic maps are considered as the generalizations  of harmonic maps and include harmonic maps as a subset, it would be very interesting to study biharmonicity of Riemannian submersions. Based on this, we will study  biharmonicity of Riemannian submersions  from 3-dimensional Sol space into a surface in the second part of the paper.
   Biharmonic Riemannian submersions were first  studied by  Oniciuc in \cite{Oni}. In \cite{WO}, the authors first introduced so-called integrability data and then used the main tool to obtain a complete classification of biharmonic Riemannian submersions from a 3-dimensional space form into a surface. In \cite{AO}, the authors studied biharmonicity of a general Riemannian submersion and obtained biharmonic equations for Riemannian submersions with one-dimensional fibers and Riemannian submersions with basic mean curvature vector fields of fibers, and they  first used the so-called integrability data to study biharmonic Riemannian submersions from $(n+1)$-dimensional spaces with one-dimensional fibers. In \cite{Ura2}, the author  studied biharmonicity a more general setting of Riemannian submersion with a $S^1$ fiber over a compact Riemannian manifold. In \cite{GO} , the authors studied  generalized harmonic morphisms and obtained many examples of biharmonic Riemannian submersions  which are maps between Riemannian manifolds that pull back local harmonic functions to local biharmonic functions.\\

 In addition to these,  we refer the readers to  the following classification results. In 2023, the authors \cite{WO1}  classified all proper biharmonic Riemannian submersions  from BCV 3-diemnsional spaces into a surface. In a recent paper \cite{WO2}, the authors also gave complete classifications of biharmonic Riemannian submersions from
 3-dimensional Berger sphere. And also,  biharmonic Riemannian submersions  from
product spaces $M^2\times\r$ to a surface have been completely classified in \cite{WO3}.\\

Recall that Sol space is one of Thurston's eight models of 3-dimensional geometry. It is the Riemannian manifold $(\r^{3}, g_{Sol})$, where the metric can be described by $g_{Sol}=e^{2z}{\rm d}x^{2}+e^{-2z}{\rm
d}y^{2}+{\rm d}z^{2}$ with respect to Euclidean coordinates on $\r^3$.\\

 First of all, one observes that it is easy to find Riemannian submersions from Sol space. For example, the projections\\
  $\pi_1: (\r^{3}, g_{Sol})\to (\r^2, e^{2z}{\rm d}x^{2}+{\rm d}z^{2}),\;\pi_1(x,y,z)=(x,z)$, and \\
  $\pi_2: (\r^{3}, g_{Sol})\to (\r^2, e^{2z}{\rm d}y^{2}+{\rm d}z^{2}),\;\pi_2(x,y,z)=(y,z)$ are both Riemannian submersions.\\

  One may wonder whether these are harmonic or biharmonic, whether there is any harmonic or biharmonic Riemannian submersions from Sol space. In this paper, we prove that there is neither harmonic nor biharmonic Riemannian submersion  $\pi:(\r^3,g_{Sol})\to (N^2,h)$ from Sol space no matter what the base space $(N^2,h)$ is.  We also prove that a Riemannian submersion $\pi:(\r^3,g_{Sol})\to (N^2,h)$ from Sol space exists only when the base space is a hyperbolic space form.

\section{Harmonic Riemannian submersions from Sol space}
In this section, we obtain a nonexistence classification  results for harmonic Riemannian submersions from Sol space to a surface.\\
Let $(\r^{3},g_{Sol})$ denote Sol space, where the metric $g_{Sol}=e^{2z}{\rm d}x^{2}+e^{-2z}{\rm
d}y^{2}+{\rm d}z^{2}$ with respect to local coordinates
$(x,y,z)$ in $\r^{3}$. We have a defined  orthonormal basis as
\begin{equation}\notag
E_{1}=e^{-z}\frac{\partial}{\partial x},\;
E_{2}=e^{z}\frac{\partial}{\partial y},
\;E_{3}=\frac{\partial}{\partial z}.
\end{equation}
With respect to this orthonormal frame, the Lie brackets and the
Levi-Civita connection are given by:
\begin{equation}\label{So1}
[E_{1},E_{2}]=0, \;[E_{2},E_{3}]=-E_{2}, \;[E_{1},E_{3}]=E_{1},
\end{equation}
\begin{equation}\label{So2}
\begin{array}{lll}
\nabla_{E_{1}}E_{1}=-E_{3},\;\nabla_{E_{1}}E_{3}=E_{1},\;
\nabla_{E_{2}}E_{2}=E_{3},\;\nabla_{E_{2}}E_{3}=-E_{2},\\
{\rm all\;other}\;\nabla_{E_{i}}E_{j}=0,\; i,j=1,2,3.\\
\end{array}
\end{equation}
One adopts the
following notation and sign convention for Riemannian curvature
operator.
\begin{equation}
 R(X,Y)Z=\nabla_{X}\nabla_{Y}Z
-\nabla_{Y}\nabla_{X}Z-\nabla_{[X,Y]}Z,\\
\end{equation}
 the Riemannian and the Ricci curvature tensors are
given by
\begin{equation}
\begin{array}{lll}
&&  R(X,Y,Z,W)=g( R(Z,W)Y,X),\\
&& {\rm Ric}(X,Y)= {\rm Trace}_{g}R=\sum\limits_{i=1}^3 R(Y, e_i, X,
e_i)=\sum_{i=1}\limits^3 \langle R( X,e_i) e_i, Y\rangle.
\end{array}
\end{equation}
A straightforward computation gives

\begin{equation}\label{So3}
\begin{array}{lll}
 R_{1212}=g(R(E_{1},E_{2})E_{2},E_{1})=1,\;
R_{1313}=g(R(E_{1},E_{3})E_{3},E_{1})=-1,\\
R_{2323}=g(R(E_{2},E_{3})E_{3},E_{2})=-1,\;\;
{\rm all\;other}\; R_{ijkl}=0, i,j,k,l=1,2,3.\\
\end{array}
\end{equation}

Let $\pi:(\r^3,g_{Sol}) \to (N^2,h)$ be a Riemannian
submersion from Sol space with an orthonormal frame  $\{e_1,\; e_2, \;e_3\}$ on $(\r^3,g_{Sol})$  and $e_3$ being  vertical.
By a treatment similar to to those used treating Remark 1 in  \cite{WO1},   we then have the following (\ref{R1})--(\ref{GCB1}) (see \cite{WO1} for details)
\begin{equation}\label{R1}
[e_1,e_3]=f_{3}e_2+\kappa_1e_3,\;
[e_2,e_3]=-f_{3}e_1+\kappa_2e_3,\;
[e_1,e_2]=f_1 e_1+f_2e_2-2\sigma e_3.
\end{equation}
where $f_1, f_2, f_3,\;\kappa_1,\;\kappa_2\;{\rm and}\; \sigma
$ are the (generalizd) integrability
data of the Riemannian submersion $\pi$. When $f_{3}=0$,  the horizontal
distribution  $\{e_1,\; e_2\}$ are basic and $ \{f_1, f_2,
\kappa_1,\;\kappa_2,\; \sigma\}$ is  the integrability
data  of the adapted frame.\\
The Levi-Civita connection is given by
\begin{eqnarray}\notag
&&\nabla_{e_{1}} e_{1}=-f_1e_2,\;\;\nabla_{e_{1}} e_{2}=f_1
e_1-\sigma e_{3},\;\;\nabla_{e_{1}} e_{3}=\sigma \notag
e_{2},\\\label{R2} &&\nabla_{e_{2}} e_{1}=-f_2 e_{2}+\sigma
e_3,\;\;\nabla_{e_{2}} e_{2}=f_2 e_{1}, \;\;\nabla_{e_{2}}
e_{3}=-\sigma e_{1},\\\notag && \nabla_{e_{3}}
e_{1}=-\kappa_1e_{3}+(\sigma-f_{3})e_{2}, \nabla_{e_{3}} e_{2}= -(\sigma-f_{3})
e_{1}-\kappa_2 e_3, \nabla_{e_{3}} e_{3}=\kappa_1 e_{1}+\kappa_2
e_2.
\end{eqnarray}

Denoting by $e_i=\sum\limits_{j=1}^{3}a_{i}^{j}E_j,\;i=1,2,3$,
using (\ref{So2}), (\ref{So3})  and (\ref{R2}), then the Jacobi identity applied to the frame $\{e_1, e_2, e_3\}$  gives
 \begin{equation}\label{Jac}
 \begin{cases}
e_3(f_1)+(\kappa_1+f_2)f_{3}-e_1(f_{3})=0,\\
e_3(f_2)+(\kappa_2-f_1)f_{3}-e_2(f_{3})=0,\\
2 e_3(\sigma)+\kappa_1f_1+\kappa_2f_2+e_2(\kappa_1)-e_1(\kappa_2)=0,
\end{cases}
\end{equation}
and the terms of the curvature tension as follows

\begin{equation}\label{RC0}
\begin{cases}
R^{M}(e_1,e_3,e_1,e_2)=-e_1(\sigma)+2\kappa_1\sigma=-2a_{2}^{3}a_{3}^{3},\\
R^{M}(e_1,e_3,e_1,e_3)=e_1(\kappa_1)+\sigma^2-\kappa_{1}^2+\kappa_2f_1=2(a_{2}^{3})^2-1,\;\\
R^{M}(e_1,e_3,e_2,e_3)=e_1(\kappa_2)-e_3(\sigma)-\kappa_{1}f_{1}-\kappa_1\kappa_2=-2a_{1}^{3}a_{2}^{3},\;\\
R^{M}(e_1,e_2,e_1,e_2)=e_1(f_2)-e_2(f_1)-f_{1}^{2}-f_{2}^{2}+2f_{3}\sigma-3\sigma^2=2(a_{3}^{3})^2-1,\\
R^{M}(e_1,e_2,e_2,e_3)=-e_2(\sigma)+2\kappa_2\sigma=2a_{1}^{3}a_{3}^{3},\\
R^{M}(e_2,e_3,e_1,e_3)=e_2(\kappa_{1})+e_3(\sigma)+\kappa_2 f_2-\kappa_1 \kappa_2=-2a_{1}^{3}a_{2}^{3},\\
R^{M}(e_2,e_3,e_2,e_3)=\sigma^{2}+e_2(\kappa_2)-\kappa_1f_2- \kappa_2^2=2(a_{1}^{3})^2-1.
\end{cases}
\end{equation}
 Gauss curvature of the base space as

\begin{equation}\label{GCB}
K^N=e_1(f_2)-e_2(f_1)-f_1^2-f_2^2+2f_{3}\sigma,
\end{equation}

\begin{equation}\label{GCB0}
e_3(K^N)=e_3[e_1(f_2)-e_2(f_1)-f_1^2-f_2^2+2f_{3}\sigma]=0.
\end{equation}
When $f_{3}=0$, then Gauss curvature of the base space becomes

\begin{equation}\label{GCB1}
K^N=e_1(f_2)-e_2(f_1)-f_1^2-f_2^2.
\end{equation}

 Now we are ready to give the following classification of harmonic Riemannian submersions from Sol space.
\begin{proposition}(see\cite{WO1})\label{PH}
A Riemannian submersion $\pi:( M^3 , g)\to (N^2,h)$ is harmonic if and only if \;$\nabla^{M}_{e_3}e_3=0$, i.e., $\kappa_1=\kappa_2=0$.
\end{proposition}
\begin{theorem}\label{HTh}
There exists no harmonic Riemannian submersion $\pi:(\r^3,g_{Sol})\to (N^2,h)$ from Sol space no matter
what $(N^2, h)$ is.
\end{theorem}

\begin{proof}
Let $\pi:(\r^3,g_{Sol})\to (N^2,h)$ be a Riemannian submersion with
an  orthonormal frame $\{e_1,\; e_2,\;e_3\}$, $e_3$  being vertical, and the (generalized) integrability data $\{ f_1, f_2, f_3, \kappa_1, \kappa_2, \sigma\}$.
By Proposition \ref{PH}, the Riemannian  submersion $\pi$ is harmonic if and only if $\kappa_1=\kappa_2=0$. Using (\ref{RC0}) and Proposition 2.2 in \cite{WO4}, we obtain
\begin{equation}\label{RCH}
\begin{cases}
-e_1(\sigma)=-2a_{2}^{3}a_{3}^{3},\\
\sigma^2=2(a_{2}^{3})^2-1,\;\\
-e_3(\sigma)=-2a_{1}^{3}a_{2}^{3}=0,\\
K^N=3\sigma^2+2(a_{3}^{3})^2-1,\\
-e_2(\sigma)=2a_{1}^{3}a_{3}^{3},\\
e_3(\sigma)=-2a_{1}^{3}a_{2}^{3}=0,\\
\sigma^{2}=2(a_{1}^{3})^2-1.
\end{cases}
\end{equation}
where $K^N=e_1(f_2)-e_2(f_1)-f_{1}^{2}-f_{2}^{2}+2f_{3}\sigma$.\\
Comparing the 2nd equation with the 7th equation of  (\ref{RCH}), we have $(a_1^3)^2= (a_2^3)^2$. However, using the 3rd equation  of  (\ref{RCH}), we get $a_1^3 a_2^3=0$ and  hence $a_1^3= a_2^3=0$. We substitute this into the 2nd equation  of  (\ref{RCH}) to have $\sigma^2=-1$, a contradiction.\\
 From which we obtain the theorem.

\end{proof}

\section{ Biharmonic Riemannian submersions from Sol space}

We state  the following proposition (\cite{WO}) which will be later used in the rest of the paper.
\begin{proposition}(see \cite{WO})\label{Pro1}
Let $\pi:(M^3,g)\to (N^2,h)$ be a Riemannaian submersion
with the adapted frame $\{e_1,\; e_2, \;e_3\}$ and the integrability
data $ f_1, f_2, \kappa_1,\;\kappa_2\;{\rm and}\; \sigma$. Then, the
Riemannaian submersion $\pi$ is biharmonic if and only if
\begin{equation}\label{pro1}
\begin{cases}
-\Delta^{M}\kappa_1-2\sum\limits_{i=1}^{2}f_i e_i(\kappa_2)-\kappa_2\sum\limits_{i=1}^{2}\left(e_i( f_i)
-\kappa_i f_i\right)+\kappa_1\left(-K^{N}+\sum\limits_{i=1}^{2}f_{i}^{2}\right)
=0,\\
-\Delta^{M}\kappa_2+2\sum\limits_{i=1}^{2}f_i e_i(\kappa_1)+\kappa_1\sum\limits_{i=1}^{2}(e_i( f_i)
-\kappa_i f_i)+\kappa_2\left(-K^{N}+\sum\limits_{i=1}^{2}f_{i}^{2}\right)=0
\end{cases}
\end{equation}
where $K^{N}=R^{N}_{1212}\circ\pi=e_1(f_2)-e_2(f_1)-f_{1}^{2}-f_{2}^{2}
$ is  Gauss curvature of Riemannian manifold $(N^2,h)$.
\end{proposition}

The following proposition was found in \cite{WO1}.
 \begin{proposition}(see \cite{WO1})\label{L1}
Let $\pi:(M^3,g)\longrightarrow (N^2,h)$ be a Riemannian submersion
from 3-manifolds  with  an orthonormal frame $\{e_1, e_2, e_3\}$ and
$ e_3$ being vertical. If $\nabla_{e_{1}}e_{1}=0$, then either $\nabla_{e_{2}}e_{2}=0$, \;or\; $\nabla_{e_{2}}e_{2}\not\equiv0$, and
the  frame $\{e_1, e_2, e_3\}$  is an adapted frame.
\end{proposition}

We will prove the important conclusion used proving our main theorem
\begin{theorem}\label{Cla}
Let $\pi:(\r^3,g_{Sol}=e^{2z}{\rm d}x^{2}+e^{-2z}{\rm
d}y^{2}+{\rm d}z^{2})\to (N^2,h)$ be a Riemannian submersion. Then, we have such an adapted  frame $\{e_1=a_1^2E_2+a_1^3E_3,\; e_2,
\;e_3\}$ of the Riemannian submersion $\pi$ with $e_3$ being vertical. Moreover, if $E_1$ is not vertical, then $\nabla_{e_2}e_2\neq0$, i.e.,\;$f_2\neq0$.
\end{theorem}
\begin{proof}
Obviously, if $ E_1$  is tangent to the fiber of the  Riemannian
submersion $\pi$ , then any basic field is of the form $e=a^2E_2+b^2E_3,\;{\rm and}\;a^2+b^2=1$.\\
From this time on, we only need  to suppose  that $ E_1$  is not vertical, i.e., $ e_3\neq\pm E_1$. Then, the vector filed $e_1=e_3\times E_1$  is  horizontal  and
 hence $\langle e_1, E_1\rangle=0$. From this, we have a defined orthonormal frame $\{e_1, \;e_2=e_3\times e_1,\;e_3\}$ on $M^3$.
If  denoting by $e_i=\sum\limits_{j=1}^{3}a_i^jE_j, i=1,2,3$, together with $\langle e_1, E_1\rangle=0$, then  $e_1$  is expressed as the form $e_1=a_1^2E_2+a_1^3E_3$ and hence \;$(a_1^2)^2+(a_1^3)^2=1$. From these, we have the following
\begin{equation}\label{zb}
\begin{array}{lll}
a_1^1=0,\;a_3^{1}\neq\pm1\;{\rm and}\; a_2^{1}\neq0.
\end{array}
\end{equation}
One can further check the following equalities  as
\begin{equation}\label{bb4}
 f_1=0,\;\nabla_{e_1}e_1=0.
\end{equation}
By a direct computation, we get
\begin{equation}\label{bb2}
\begin{array}{lll}
\nabla_{e_{1}} e_{1}=\nabla_{e_{1}}(\sum\limits_{i=1}^{3}a_1^iE_i)
=\sum\limits_{i=1}^{3}e_1(a_1^i)E_i+\sum\limits_{i,j=1}^{3}a_1^ja_1^i\nabla_{E_j}E_i.
\end{array}
\end{equation}
However, using (\ref{R2}), the above has  another expression as
\begin{equation}\label{bb1}
\nabla_{e_{1}} e_{1}=-f_1e_{2}=-f_1\sum\limits_{i=1}^{3}a_2^iE_i.
\end{equation}

By equating  (\ref{bb2}) and (\ref{bb1}) and  comparing the coefficient of $E_1$, we obatin
 \begin{equation}\label{bb3}
\begin{array}{lll}
-f_1 a_2^1=\langle-f_1\sum\limits_{i=1}^{3}a_1^iE_i,E_1\rangle=\langle\nabla_{e_{1}} e_{1},E_1\rangle\\
=\langle\sum\limits_{i=1}^{3}e_1(a_1^i)E_i+\sum\limits_{i,j=1}^{3}a_1^ja_1^i\nabla_{E_j}E_i,E_1\rangle=e_1(a_1^1)=0,
\end{array}
\end{equation}
which has been used (\ref{So2}) and $a_1^1=0$. This leads to  $f_1=0$ for $a_2^1\neq0$, and hence (\ref{bb4}) holds.\\
Applying (\ref{So2}), (\ref{R2}) and $a_1^1=f_1=0$ and  a further computation  similar to those used calculating  (\ref{bb2})--(\ref{bb3}) gives
\begin{equation}\label{thb2}
\begin{cases}
e_1(a_{1}^{2})=a_1^2a_{1}^{3},\\
e_1(a_{1}^{3})=-(a_{1}^{2})^2,\\

e_1(a_{2}^{1})=-\sigma a_3^1,\\
e_1(a_{2}^{2})=a_1^2a_2^3-\sigma a_{3}^{2},\\
e_1(a_{2}^{3})=-a_1^2a_2^2-\sigma a_{3}^{3},\\
e_1(a_{3}^{1})=\sigma a_2^1,\\
e_1(a_{3}^{2})=a_1^2a_3^3+\sigma a_{2}^{2},\\
e_1(a_{3}^{3})=-a_1^2a_3^2+\sigma a_{2}^{3},\\
f_2a_2^1=-a_1^3a_2^1+\sigma a_3^1=-a_3^2+\sigma a_3^1,\\
e_2(a_{3}^{1})=-a_2^1a_3^3,\\
e_2(a_{2}^{1})=-a_2^1a_2^3,\\
\kappa_1a_3^1=(\sigma -f_{3}) a_2^1-a_1^3a_3^1=(\sigma -f_{3}) a_2^1+a_2^2.

\end{cases}
\end{equation}

Since $\nabla_{e_{1}}e_{1}=0$, we conclude from Proposition \ref{L1} to have either $\nabla_{e_{2}}e_{2}\neq0$,  and the
frame $\{e_1, e_2, e_3\}$  is adapted  to the Riemannian submersion $\pi$; or \;$\nabla_{e_{2}}e_{2}=0$. Now, we just need to consider the latter case
$\nabla_{e_{2}}e_{2}=0$,  i.e., $f_2=0$.  From these, one has the following
\begin{equation}\label{zb1}
a_1^1=f_1=f_2=0.
\end{equation}
Then,  (\ref{RC0}) turns into
\begin{equation}\label{RC1}
\begin{cases}
-e_1(\sigma)+2\kappa_1\sigma=-2a_{2}^{3}a_{3}^{3},\\
e_1(\kappa_1)+\sigma^2-\kappa_{1}^2=2(a_{2}^{3})^2-1,\;\\
e_1(\kappa_2)-e_3(\sigma)-\kappa_1\kappa_2=-2a_{1}^{3}a_{2}^{3},\;\\
2f_{3}\sigma-3\sigma^2=2(a_{3}^{3})^2-1,\\
-e_2(\sigma)+2\kappa_2\sigma=2a_{1}^{3}a_{3}^{3},\\
e_2(\kappa_{1})+e_3(\sigma)-\kappa_1 \kappa_2=-2a_{1}^{3}a_{2}^{3},\\
\sigma^{2}+e_2(\kappa_2)-\kappa_2^2=(2a_{1}^{3})^2-1.
\end{cases}
\end{equation}

We now show that the latter case (i.e., $a_1^1=f_1=f_2=0$,  $a_3^1\neq\pm1$ and $a_2^1\neq0$) can not happen by considering the following two cases:\\

Case I: $a_3^1=0,\;f_2=0$. In this case, since $a_1^1=0$, we have $a_2^1=\pm1$ and hence $a_2^3=0$.  By the 9th equation  of (\ref{thb2}),
one easily sees that $a_3^2= a_3^1=0$ and hence $a_3^3=\pm1$. This leads to $a_1^3=0$ and $a_1^2=\pm1$. Substituting this into the 6th equation  of  (\ref{thb2}), we have
 $\sigma=0$. However, we substitute $\sigma=0$ and  $a_3^3=\pm1$ into the 4th equation  of  (\ref{RC1}) to find  $0=1$, a contradiction.\\

Case II: $a_3^1\neq0,\;\pm1$ and $f_2=0$. In this case, since $a_1^1=0$, we then have $a_2^1\neq0,\pm1$. Substitute $f_2=0$ into the 9th equation of (\ref{thb2}) to have
\begin{equation}\label{s1}
\begin{array}{lll}
a_3^2=\sigma a_3^1.
\end{array}
\end{equation}
 Applying $e_1$ to both sides the 12th equation of  (\ref{thb2}), we get
\begin{equation}\label{s2}
\begin{array}{lll}
e_1(\kappa_1) a_3^1+\kappa_1e_1(a_3^1)=e_1(\sigma)a_2^1+\sigma e_1(a_2^1)-e_1(f_3)a_2^1-f_3 e_1(a_2^1)+e_1( a_2^2),
\end{array}
\end{equation}
which can be rewritten as
\begin{equation}\label{s3}
\begin{array}{lll}
e_1(\kappa_1) a_3^1+\kappa_1e_1(a_3^1)-e_1(\sigma)a_2^1-\sigma e_1(a_2^1)+e_1(f_3)a_2^1+f_3 e_1(a_2^1)-e_1( a_2^2)=0.
\end{array}
\end{equation}
Using the 3rd, the 4th, the 6th equation of (\ref{thb2}), the 1st, the 2nd equation of (\ref{RC1}) and  the 1st equation of (\ref{Jac}), a straightforward computation gives
\begin{equation}\label{s4}
\begin{array}{lll}
0=(\kappa_1^2-\sigma^2+2(a_2^3)^2-1) a_3^1+\kappa_1\sigma a_2^1\\
-(2\kappa_1\sigma+2a_2^3a_3^3)a_2^1+\sigma^2 a_3^1+\kappa_1f_3a_2^1-f_3 \sigma a_3^1-a_1^2a_2^3+\sigma a_3^2\\
=\kappa_1^2a_3^1-\sigma^2a_3^1+2a_3^1(a_2^3)^2-a_3^1+\kappa_1\sigma a_2^1\\
-2\kappa_1\sigma a_2^1-2a_2^3a_3^3a_2^1+\sigma^2 a_3^1+\kappa_1f_3a_2^1-f_3 \sigma a_3^1-a_1^2a_2^3+\sigma a_3^2.
\end{array}
\end{equation}
One substitutes the 12th equation of (\ref{thb2}) and (\ref{s1}) into (\ref{s3}), together with $a_2^1a_2^3+a_3^1a_3^3=0$ and $a_2^2=a_1^3a_3^1$, to compute the following
\begin{equation}\label{s05}
\begin{array}{lll}
0
=\kappa_1(\sigma a_2^1-f_3a_2^1+a_2^2)-\sigma^2a_3^1+2a_3^1(a_2^3)^2-a_3^1+\kappa_1\sigma a_2^1\\
-2\kappa_1\sigma a_2^1+2a_3^1(a_3^3)^2+\sigma^2 a_3^1+\kappa_1f_3a_2^1-f_3 \sigma a_3^1-a_1^2a_2^3+\sigma^2 a_3^1\\
=\kappa_1a_2^2+2a_3^1[(a_2^3)^2+(a_3^3)^2]-a_3^1-f_3 \sigma a_3^1-a_1^2a_2^3+\sigma^2 a_3^1\\
=\kappa_1a_3^1a_1^3+2a_3^1[1-(a_1^3)^2]-a_3^1-f_3 \sigma a_3^1-a_1^2a_2^3+\sigma^2 a_3^1\\
=(\sigma a_2^1-f_3a_2^1+a_2^2)a_1^3+2a_3^1-2a_3^1(a_1^3)^2-a_3^1-f_3 \sigma a_3^1-a_1^2a_2^3+\sigma^2 a_3^1\\
=\sigma a_2^1a_1^3-f_3a_2^1a_1^3+a_2^2a_1^3-2a_3^1(a_1^3)^2+a_3^1-f_3 \sigma a_3^1-a_1^2a_2^3+\sigma^2 a_3^1\\
=2\sigma^2 a_3^1-2f_3\sigma a_3^1-2a_3^1(a_1^3)^2,
\end{array}
\end{equation}
the last equality holds by using the fact $a_2^1a_1^3=a_3^2$, $a_3^2=\sigma a_3^1$  and  $a_3^1=a_1^2a_2^3-a_1^3a_2^2$.\\
Since $a_3^1\neq0$, then  (\ref{s05}) becomes
\begin{equation}\label{s5}
\begin{array}{lll}
2\sigma^2-2f_3\sigma -2(a_1^3)^2=0
\end{array}
\end{equation}
Substituting the 4th equation of (\ref{RC1}) into  (\ref{s5}), together with $(a_1^3)^2+(a_2^3)^2+(a_3^3)^2=1$, and  simplifying the resulting equation, we get
\begin{equation}\label{s6}
\begin{array}{lll}
\sigma^2=2(a_2^3)^2-1.
\end{array}
\end{equation}
This implies
\begin{equation}\label{s7}
\begin{array}{lll}
\sigma^2(a_3^1)^2=2(a_2^3a_3^1)^2-(a_3^1)^2.
\end{array}
\end{equation}
Since $\sigma a_3^1=a_3^2$, the above equation  is equivalent to
\begin{equation}\label{s8}
\begin{array}{lll}
2(a_2^3a_3^1)^2=(a_3^1)^2+(a_3^2)^2=1-(a_3^3)^2,
\end{array}
\end{equation}
or,
\begin{equation}\label{s9}
\begin{array}{lll}
2(a_2^3a_3^1)^2+(a_3^3)^2-1=0.
\end{array}
\end{equation}
On the other hand,  let $\theta$, $\alpha$ denote angles between $e_1$ and $E_2$,  between $e_3$ and $E_1$, respectively,  since $a_1^1=0$,  we have
 \begin{equation}\label{CR1}
\begin{cases}
 e_1=\cos\theta E_2+\sin\theta E_3, \\
 e_2=\sin\alpha E_1-\sin\theta\cos\alpha E_2+\cos\theta\cos\alpha E_3, \\
 e_3=\cos\alpha E_1+\sin\theta\sin\alpha E_2-\cos\theta\sin\alpha E_3,
\end{cases}
\end{equation}
where $a_1^2=\cos\theta,\;a_1^3=\sin\theta,\;a_2^1=\sin\alpha,\;a_2^2=-\sin\theta\cos\alpha,\;a_2^3=\cos\theta\cos\alpha,\;a_3^1=\cos\alpha,\;a_3^2=\sin\theta\sin\alpha $\;and $a_3^3=-\cos\theta\sin\alpha$.\\
Using the 1st  and the 3rd equation of  (\ref{thb2}), it is not difficult to check the following
\begin{equation}\label{s10}
\begin{array}{lll}
e_1(\alpha)=-\sigma,\;e_1(\theta)=-\cos\theta.
\end{array}
\end{equation}
Since $a_2^3=\cos\theta\cos\alpha,\;a_3^1=\cos\alpha$ and $\;a_3^3=-\cos\theta\sin\alpha $, then (\ref{s9}) turns into
\begin{equation}\label{s11}
\begin{array}{lll}
\cos^2\theta(2\cos^4\alpha +\sin^2\alpha) -1=0.
\end{array}
\end{equation}
Note that $\theta$ is angle between $e_1$ and $E_2$,  but  $\alpha$  angle between $e_3$ and $E_1$,  then the two functions: $ \cos^2\theta$,\;\;$2\cos^4\alpha +\sin^2\alpha$\; are linearly independent. Then, Eq. (\ref{s11}) implies taht $\theta$ and $\alpha$ have to be constants, and hence $ \cos\theta=0$  since  (\ref{s10}). Substituting this into  Eq. (\ref{s11}). we have $-1=0$, a contradiction.\\
Summarizing all results in the above cases, the theorem follows.

\end{proof}

\begin{remark}\label{re1}
Let $\pi:(\r^3,g_{Sol}=e^{2z}{\rm d}x^{2}+e^{-2z}{\rm
d}y^{2}+{\rm d}z^{2})\to (N^2,h)$ be a Riemannian submersion
 with $e_3$ being vertical. If $e_3\neq \pm E_1$, i.e.,\;$a_3^1\neq\pm1$, by Theorem \ref{Cla},
one can choose such an adapted frame $\{e_1=a_1^2E_2+a_1^3E_3,\; e_2,
\;e_3\}$ to $\pi$ and $f_2\neq0$. From these, the case corresponding to  $a_1^1=f_1=f_3=0$,\;$f_2\neq0$, $a_3^1\neq\pm1$ and $a_2^1\neq0$. Clearly, this implies that
the case $a_3^1\neq\pm1$, and $a_1^1=f_1=f_2=0$ can not happen.
\end{remark}
\begin{theorem}\label{cl}
A Riemannian submersion $\pi:(\r^3,g_{Sol}=e^{2z}{\rm d}x^{2}+e^{-2z}{\rm
d}y^{2}+{\rm d}z^{2})\to (N^2,h)$  from  Sol space exists only in  $(\r^3,g_{Sol})\to H^2$ with Gauss curvature of the base space $K^N=-1$.
\end{theorem}
\begin{proof}
Let $\pi:(\r^3,g_{Sol}=e^{2z}{\rm d}x^{2}+e^{-2z}{\rm
d}y^{2}+{\rm d}z^{2})\to (N^2,h)$ be a Riemannian submersion with $e_3$ being vertical.  For the above notations and signs, we just need to
consider the two cases $a_3^1=0$ or $a_3^3=\pm1$. We use the proof by contradiction to obtain the theorem. We now assume $a_3^1\neq0$ and  $a_3^3\neq\pm1$.
It follows from Theorem \ref{Cla} and Remark \ref{re1} that there exists  such an adapted frame $\{e_1=a_1^2E_2+a_1^3E_3,\; e_2,
\;e_3\}$ to $\pi$, and hence the following hold
\begin{equation}\label{c1}
\begin{array}{lll}
a_1^1=f_1=f_3=0,\;f_2\neq0,\;a_2^1\neq0,\pm1\;{\rm and}\;a_3^1\neq0,\pm1.
\end{array}
\end{equation}
Then, (\ref{RC0}) becomes as
\begin{equation}\label{RC2}
\begin{cases}
-e_1(\sigma)+2\kappa_1\sigma=-2a_{2}^{3}a_{3}^{3},\\
e_1(\kappa_1)+\sigma^2-\kappa_{1}^2=2(a_{2}^{3})^2-1,\;\\
e_1(\kappa_2)-e_3(\sigma)-\kappa_1\kappa_2=-2a_{1}^{3}a_{2}^{3},\;\\
e_1(f_2)-f_{2}^{2}-3\sigma^2=2(a_{3}^{3})^2-1,\\
-e_2(\sigma)+2\kappa_2\sigma=2a_{1}^{3}a_{3}^{3},\\
e_2(\kappa_{1})+e_3(\sigma)+\kappa_2 f_2-\kappa_1 \kappa_2=-2a_{1}^{3}a_{2}^{3},\\
\sigma^{2}+e_2(\kappa_2)-\kappa_1f_2-\kappa_2^2=2(a_{1}^{3})^2-1.\\
\end{cases}
\end{equation}

We  apply $e_1$ to both sides the 12th equation of (\ref{thb2}), together with $f_3=0$, to get
\begin{equation}\label{ss3}
\begin{array}{lll}
e_1(\kappa_1) a_3^1+\kappa_1e_1(a_3^1)-e_1(\sigma)a_2^1-\sigma e_1(a_2^1)-e_1( a_2^2)=0.
\end{array}
\end{equation}
Using the 3rd, the 4th, the 6th equation of (\ref{thb2}), the 1st, the 2nd equation of (\ref{RC2}) and  the 1st equation of (\ref{Jac}), a straightforward computation gives
\begin{equation}\label{ss4}
\begin{array}{lll}
0=(\kappa_1^2-\sigma^2+2(a_2^3)^2-1) a_3^1+\kappa_1\sigma a_2^1
-(2\kappa_1\sigma+2a_2^3a_3^3)a_2^1+\sigma^2 a_3^1-a_1^2a_2^3+\sigma a_3^2\\
=\kappa_1^2a_3^1-\sigma^2a_3^1+2a_3^1(a_2^3)^2-a_3^1+\kappa_1\sigma a_2^1
-2\kappa_1\sigma a_2^1-2a_2^3a_3^3a_2^1+\sigma^2 a_3^1-a_1^2a_2^3+\sigma a_3^2.
\end{array}
\end{equation}
Substituting  the 12th equation of  (\ref{thb2}) into (\ref{ss3}), together with $a_2^1a_2^3+a_3^1a_3^3=0$, $f_3=0$  and $a_2^2=a_1^3a_3^1$, a direct computation gives
\begin{equation}\label{ss5}
\begin{array}{lll}
0
=\kappa_1(\sigma a_2^1+a_2^2)-\sigma^2a_3^1+2a_3^1(a_2^3)^2-a_3^1+\kappa_1\sigma a_2^1
-2\kappa_1\sigma a_2^1+2a_3^1(a_3^3)^2+\sigma^2 a_3^1-a_1^2a_2^3+\sigma^2 a_3^1\\
=\kappa_1a_2^2+2a_3^1[(a_2^3)^2+(a_3^3)^2]-a_3^1-a_1^2a_2^3+\sigma^2 a_3^1\\
=\kappa_1a_3^1a_1^3+2a_3^1[1-(a_1^3)^2]-a_3^1-a_1^2a_2^3+\sigma^2 a_3^1\\
=(\sigma a_2^1+a_2^2)a_1^3+2a_3^1-2a_3^1(a_1^3)^2-a_3^1-a_1^2a_2^3+\sigma^2 a_3^1\\
=\sigma a_2^1a_1^3+a_2^2a_1^3-2a_3^1(a_1^3)^2+a_3^1-a_1^2a_2^3+\sigma^2 a_3^1\\
=\sigma^2 a_3^1+\sigma a_3^2-2a_3^1(a_1^3)^2,
\end{array}
\end{equation}
the last equality holds for using the fact $a_2^1a_1^3=a_3^2$ and  $a_3^1=a_1^2a_2^3-a_1^3a_2^2$.\\

Since $a_1^1=0$, one can assume that
 \begin{equation}\label{CR1}
\begin{cases}
 e_1=\cos\theta E_2+\sin\theta E_3, \\
 e_2=\sin\alpha E_1-\sin\theta\cos\alpha E_2+\cos\theta\cos\alpha E_3, \\
 e_3=\cos\alpha E_1+\sin\theta\sin\alpha E_2-\cos\theta\sin\alpha E_3,
\end{cases}
\end{equation}
where $a_1^2=\cos\theta,\;a_1^3=\sin\theta,\;a_2^1=\sin\alpha,\;a_2^2=-\sin\theta\cos\alpha,\;a_2^3=\cos\theta\cos\alpha,\;a_3^1=\cos\alpha,\;a_3^2=\sin\theta\sin\alpha $\;and $a_3^3=-\cos\theta\sin\alpha$.\\
We applying  the 1st and the 3rd equation of (\ref{thb2}) to see that
\begin{equation}\label{ss7}
\begin{array}{lll}
e_1(\alpha)=-\sigma,\;e_1(\theta)=-\cos\theta.
\end{array}
\end{equation}
Since $a_1^3=\sin\theta,\;a_3^1=\cos\alpha$ and $\;a_3^2=\sin\theta\sin\alpha $, then  (\ref{ss5}) becomes
\begin{equation}\label{ss8}
\begin{array}{lll}
\sigma^2\cos\alpha +\sigma\sin\theta\sin\alpha -2\cos\alpha\sin^2\theta=0.
\end{array}
\end{equation}

One solves the  above equation to obtain
\begin{equation}\label{ss9}
\begin{array}{lll}
\sigma=\sin\theta\varphi(\alpha),
\end{array}
\end{equation}
where $\varphi(\alpha)=\frac{-\sin\alpha\pm\sqrt{1+7\cos^2\alpha}}{\cos\alpha}.$\\

Substituting $a_2^1=\sin\alpha,\;a_2^2=-\sin\theta\cos\alpha$, $a_3^1=\cos\alpha$, $f_3=0$  and (\ref{ss9})  into the 12th equation of (\ref{thb2}),  we have
\begin{equation}\label{ss10}
\begin{array}{lll}
\kappa_1=\sin\theta\psi(\alpha),
\end{array}
\end{equation}
where denote by $\psi(\alpha)=-\cos\alpha+\sin\alpha\varphi(\alpha)$ and $\varphi(\alpha)=\frac{-\sin\alpha\pm\sqrt{1+7\cos^2\alpha}}{\cos\alpha}.$\\

We substitute (\ref{ss9}) and (\ref{ss10}) into the 1st equation of  (\ref{RC2}) to have
\begin{equation}\label{ss11}
\begin{array}{lll}
e_1(\sigma)=2\kappa_1\sigma+2a_2^3a_3^3
=2\sin^2\theta\varphi(\alpha)\psi(\alpha)-2\cos^2\theta\cos\alpha\sin\alpha,
\end{array}
\end{equation}
where $\psi(\alpha)=-\cos\alpha+\sin\alpha\varphi(\alpha)(\alpha)$ and $\varphi(\alpha)=\frac{-\sin\alpha\pm\sqrt{1+7\cos^2\alpha}}{\cos\alpha}.$\\
On the other hand, substitute  (\ref{ss9}) into the left-hand side of  (\ref{ss11}) to  compute as
\begin{equation}\label{ss12}
\begin{array}{lll}
e_1(\sigma)=e_1\left(\sin\theta\varphi(\alpha)\right)
=\cos\theta\varphi(\alpha) e_1(\theta)+\sin\theta\varphi'(\alpha) e_1(\alpha)\\
=-\cos^2\theta\varphi(\alpha) -\sin^2\theta\varphi'(\alpha)\varphi(\alpha),
\end{array}
\end{equation}
the last holds by using (\ref{ss7}).\\

Comparing (\ref{ss11}) with (\ref{ss12}), we deduce
\begin{equation}\label{ss13}
\begin{array}{lll}
\sin^2\theta\{2\varphi(\alpha)\psi(\alpha)+\varphi'(\alpha)\varphi(\alpha)\}+\cos^2\theta\{-2\cos\alpha\sin\alpha+\varphi(\alpha)\}=0.
\end{array}
\end{equation}
Solving the above equation, we  obtain a contradiction. Indeed,  if $\theta$ is a constant, by the 2nd equation of (\ref{ss7}), we have $\cos\theta=0$ and $\sin\theta=\pm1$.  Substituting  this into the  above equation  we have
\begin{equation}\label{ss14}
\begin{array}{lll}
2\varphi(\alpha)\psi(\alpha)+\varphi'(\alpha)\varphi(\alpha)=0,\\
\end{array}
\end{equation}
this implies $\alpha$ being constant and hence $\sigma=0$ together with the 1st equation of (\ref{ss7}). Hence, using (\ref{ss8}), we get  $a_3^1=\cos\alpha=0$\; since \;$\sin\theta=\pm1$, a contradiction.\\

If $\theta\neq {\rm constant}$,  but the two functions $ \sin^2\theta$,\;\;$\cos^2\theta$ are linearly independent,  then  (\ref{ss13}) means
\begin{equation}\label{ss14}
\begin{array}{lll}
2\varphi(\alpha)\psi(\alpha)+\varphi'(\alpha)\varphi(\alpha)=0,\;{\rm and}\;
-2\cos\alpha\sin\alpha+\varphi(\alpha)=0,
\end{array}
\end{equation}
where $\psi(\alpha)=-\cos\alpha+\sin\alpha\varphi(\alpha)$ and $\varphi(\alpha)=\frac{-\sin\alpha\pm\sqrt{1+7\cos^2\alpha}}{\cos\alpha}.$\\
Clearly, the second equation of the above equation implies $\alpha$ being a constant and hence $\sigma=0$ together with the 1st equation of (\ref{ss7}).  Moreover, by  (\ref{ss8}) and $a_3^1=\cos\alpha\neq0$, one finds that  $\sin\theta=0$ and hence $\theta$ is a constant contradicting the assumption $\theta\neq {\rm constant}$.  From these, when $a_3^1\neq\pm1,0$, there is no a Riemannaian submersion $\pi:(\r^3,g_{Sol})\to (N^2,h)$
from Sol space no matter what $(N^2,h)$ is.\\
In addition, if $a_3^1=0$, we have $a_2^1=\pm1$ since $a_1^1=0$, in this case, a straightforward computation similar to those used computing  Case II in Theorem \ref{Sol} gives $f_1=f_3=0$, $f_2=-1$ and hence Gauss curvature of the base space $K^N=e_1(f_2)-f_2^2=-1$; if $a_3^1=\pm1$, we have $a_2^1=0$ since $a_1^1=0$, in this case, a direct calculation  similar to those used calculating  Case I in  Theorem \ref{Sol} gives $f_1=0$, $f_2=1$ and hence Gauss curvature of the base space $K^N=e_1(f_2)-f_2^2=-1$. Clearly, This implies that the a Riemannaian submersion  $\pi:(\r^3,g_{Sol})\to (N^2,h)$ exists only in $(\r^3,g_{Sol})\to H^2$ with Gauss curvature of the base space $K^N=-1$.\\
From which we obtain the theorem.

\end{proof}

\begin{theorem}\label{Sol}
There exists no biharmonic  Riemannaian submersion  $\pi:(\r^3,g_{Sol})\to (N^2,h)$
no matter what $(N^2,h)$ is.
\end{theorem}
\begin{proof}
Let $\nabla$  denote the Levi-Civita connection on Sol space $(\r^3,g_{Sol})$ with  an orthonormal frame $\{e_1,\; e_2, \;e_3\}$ and $e_3$ being  vertical. We denote by
$e_i=\sum\limits_{j=1}^{3}a_{i}^{j}E_j,\;i=1,2,3$. To complete the proof of the theorem, from Theorem \ref{cl}, we only  discuss   biharmonicity of a Riemannaian submersion  $\pi:(\r^3,g_{Sol})\to H^2$. Furthermore, from the proof of Theorem \ref{cl}, we only need to consider the two cases $a_3^1=\pm1$  or $a_3^1=0$. \\
Case I: $a_3^{1}=\pm1$.
In this case,  one sees that $e_3=\pm E_1$ and hence take  an orthogonal frame
$\{e_1=E_3,\; e_2=E_2, \;e_3= -E_1\}$  on $(\r^3,g_{Sol})$
with $e_3$  being  vertical. A direct computation using (\ref{So1}) and (\ref{So2}) gives
\begin{equation}\notag
\begin{array}{lll}
[e_1,e_2]=e_2,\;[e_1,e_3]=-e_3,\;[e_2,e_3]=0,\\
\nabla_{e_{2}} e_{1}=- e_{2},\;\;\nabla_{e_{2}} e_{2}= e_{1},\; \nabla_{e_{3}}
e_{1}=e_{3},\;\nabla_{e_{3}} e_{3}=-e_{1},\\
{\rm all\;other}\; \nabla_{e_{i}} e_{j}= 0,\;i,j=1,2,3.
\end{array}
\end{equation}
It follows that the ( generalized) integrability data $f_1=f_{3}=\kappa_2=\sigma=0, \kappa_1=-f_2=-1$ and hence $\{e_1=E_3,\; e_2=E_2, \;e_3= E_1\}$
is actually  adapted  to $\pi$ with $e_3$ being vertical. Then,  biharmonic equation (\ref{pro1}) reduces to
\begin{equation}\label{be1}
\begin{array}{lll}
 \Delta\kappa_1-\kappa_1\{-K^N+f_2^2\}=0.
\end{array}
\end{equation}
However, the left-hand term of  (\ref{be1}) can be computed as
\begin{equation}\label{be2}
\begin{array}{lll}
 \Delta\kappa_1-\kappa_1\{-K^N+f_2^2\}
 =\sum\limits_{i=1}^{3}\left(e_ie_i(\kappa_1)-\nabla_{e_i}e_i(\kappa_1)\right)-\kappa_1\{-e_1(f_2)+2f_2^2\}\\
 =0+1\times2=2\neq0.
\end{array}
\end{equation}
Therefore, the Riemannian submersion $\pi$ is not biharmonic in this case.\\

Case II: $a_3^1=0.$ In this case, we have $a_2^1=\pm 1$ since $a_1^1=0$. Then,  we can take an orthonormal frame $\{e_1=E_3,\; e_2=E_1,,
\;e_3=E_2\}$  with $e_3$ being vertical.  A direct computation using (\ref{So1}) and  (\ref{So2}) gives
\begin{equation}\label{th6}
\begin{array}{lll}
[e_1,e_2]=-e_2,\;[e_1,e_3]=e_3,\;[e_2,e_3]=0,\\
\nabla_{e_{2}} e_{1}=-e_{2},\;\nabla_{e_{2}} e_{2}=- e_{1},\;
\nabla_{e_{3}}
e_{1}=-e_{3},\; \nabla_{e_{3}} e_{3}=e_{1},\\
{\rm all\;other}\; \nabla_{e_{i}} e_{j}= 0,\;i,j=1,2,3.
\end{array}
\end{equation}
This follows that the ( generalized) integrability data $f_1=f_{3}=\kappa_2=\sigma=0, \kappa_1=-f_2=1$ and hence $\{e_1=E_3,\; e_2=E_1,,
\;e_3=E_2\}$ becomes adapted to $\pi$ with $e_3$ being vertical.
 Substituting this into biharmonic equation (\ref{pro1}) and a direct computation, we have
\begin{equation}\label{be3}\notag
\begin{array}{lll}
 0=\Delta\kappa_1-\kappa_1\{-K^N+f_2^2\}\\
 =\sum\limits_{i=1}^{3}\left(e_ie_i(\kappa_1)-\nabla_{e_i}e_i(\kappa_1)\right)-\kappa_1\{-e_1(f_2)+2f_2^2\}
 =0+1\times2=2,
\end{array}
\end{equation}
 which is a contradiction. Thus, the Riemannian submersion $\pi$ is not biharmonic.\\
Summarizing all results in the above cases we obtain the theorem.
 \end{proof}

\begin{remark}\label{re2}
 We would like to point out that, with respect to local coordinates, a Riemannaian submersion $\pi:(\r^3,g_{Sol})\to H^2$  can be locally expressed as  the following  (up to equivalence):\\
$(a)$: the Riemannaian submersion
 \begin{equation}\notag
\begin{array}{lll}
\pi:(\r^3,g_{Sol}= e^{2z}{\rm d}x^{2}+e^{-2z}{\rm d}y^{2}+{\rm
 d}z^{2})\to (\r^2, e^{-2z}{\rm d}y^{2}+{\rm
 d}z^{2}),\;\pi(x,y,z)=(y,z),
\end{array}
\end{equation}
or,\\
$(b)$: the Riemannaian submersion
\begin{equation}\notag
\begin{array}{lll}
\pi:(\r^3,g_{Sol}= e^{2z}{\rm d}x^{2}+e^{-2z}{\rm d}y^{2}+{\rm
 d}z^{2})\to (\r^2, e^{2z}{\rm d}x^{2}+{\rm
 d}z^{2}),\;
\pi(x,y,z)=(x,z).
\end{array}
\end{equation}
By Theorem \ref{HTh} and Theorem \ref{Sol}, these Riemannaian submersions are neither harmonic nor biharmonic.
\end{remark}

As a consequence of Theorem \ref{HTh} and Theorem \ref{Sol}, we state the following fact
\begin{corollary}\label{CSol2}
 Any Riemannaian submersion $\pi:(\r^3,g_{Sol})\to (N^2,h)$
from Sol space to a surface is neither harmonic nor biharmonic.
\end{corollary}

Although there is no (harmonic) biharmonic Riemannaian submersion from Sol space to a surface,  there exist many (harmonic) biharmonic maps $(\r^3,g_{Sol})\to (N^2,h)$ which are not  Riemannaian submersions.
\begin{example}
The maps $\phi:(\r^3,g_{Sol}= e^{2z}{\rm d}x^{2}+e^{-2z}{\rm d}y^{2}+{\rm
 d}z^{2})\to (\r^2,du^2+dv^2)$,\\
$\phi(x,y,z)=(u,v)=(y,Az^3+Bz^2+Cz+D)$
are biharmonic,  where $A, B, C, D$ are constants. In particular, when $A^2+B^2> 0$, this family of maps are proper biharmonic.
 Note that these maps are not  Riemannaian submersions.
\end{example}

\end{document}